\documentclass[letterpaper, onecolumn, 12pt]{ieeeconf}  

\IEEEoverridecommandlockouts                              

\usepackage{amsmath} 
\usepackage{amssymb}  
\usepackage{algorithm}
\usepackage{algorithmic}

\usepackage{amsfonts}

\usepackage{amsthm}
\usepackage{bm}

\newtheorem{assumption}{Assumption}
\newtheorem{theorem}{Theorem}
\newtheorem{proposition}{Proposition}
\newtheorem{lemma}{Lemma}

\title{\LARGE \bf
Fully Asynchronous Push-Sum With Growing Intercommunication Intervals\authorrefmark{1}\thanks{*Research was partially supported by the NSF under grants CNS-1645681, CCF-1527292, IIS-1237022, and IIS-1724990, by the ARO under grant W911NF-12-1-0390, by the NIH under grant 1UL1TR001430, by the Boston University Digital Health Initiative, and by the joint Boston University and Brigham \& Women's Hospital program in Engineering and Radiology.}}

\author{Alex Olshevsky, Ioannis Ch. Paschalidis$^\dag$ and Artin Spiridonoff$^\ddag$
\thanks{$^\dag$ A. Olshevsky and I. Paschalidis are with the ECE Department, Boston University, {\tt\small \{alexols,yannisp\}@bu.edu}.}
\thanks{$^\ddag$ A. Spiridonoff is at the Division of Systems Engineering, Boston University, {\tt\small artin@bu.edu}.}
}

\begin{document}

\maketitle
\thispagestyle{empty}
\pagestyle{empty}

\begin{abstract}

We propose an algorithm for average consensus over a directed graph which is both fully asynchronous and robust to unreliable communications. We show its convergence to the average,
while allowing for slowly growing but potentially unbounded communication failures.

\end{abstract}

\section{Introduction}

Consider a set of agents, whose goal is to reach consensus by exchanging information locally with their neighbors through a directed graph.
There is a large body of work on consensus algorithms.
Ordinary consensus has been shown to converge asymptotically under various scenarios such as growing intercommunicating intervals \cite{lorenz2011convergence}, presence of delays and/or unbounded intercommunication intervals \cite{Blondel2005}. 
Another problem of interest for which extensive research has been carried out is \textit{average} consensus. While most related works study asymptotic convergence, \cite{charalambous2015distributed} studies average consensus in a finite number of steps.
Push-sum is one of the many algorithms for average consensus that was first proposed by \cite{kempe2003gossip}. This algorithm has been widely used to develop protocols that reach average consensus, under different assumptions and scenarios; such as the presence of bounded delays \cite{hadjicostis2014average}, time varying graphs \cite{hadjicostis2012average}\cite{Rezaeinia2017}, or asynchronous communication \cite{benezit2010weighted}.

Since reliable communication is a very restrictive assumption in network applications, or expensive to enforce, recent work has considered algorithms that reach consensus in a setting where communication between agents is unreliable. 
While in this case, push-sum might not converge to average, exponential convergence still holds and the error between the final value and the true average can be characterized \cite{DBLP:journals/corr/GerencserH15}.
In \cite{Vaidya}, Vaidya et al. introduce the technique of running sums (counters) and modify push-sum to overcome possible packet drops and imprecise knowledge of the network in a synchronous communication setting. They prove \textit{almost surely} convergence of their algorithms using weak ergodicity.
Inspired by \cite{Vaidya}, \cite{Schenato} takes this further and develops an asynchronous algorithm for average consensus, which is robust to unreliable communication. This algorithm uses a \textit{broadcast asymmetric} communication protocol; that is, at each iteration only one node is allowed to wake up and transmit information to its neighbors. Exponential convergence of this algorithm is proved under bounded consecutive link failures and nodes' update delays.
 
Consensus and average consensus have a lot of application in other algorithms as well; they can be used as a building block to develop distributed optimization algorithms \cite{tsianos2012push}\cite{varagnolo2016newton}. For example, in \cite{bof2017newton} the authors use a robust version of push-sum as a building block to develop an asynchronous Newton-based distributed optimization algorithm, robust to packet losses.

A lot of available works in the literature assume bounded intercommunication intervals; which motivated us to study and explore sufficient connectivity conditions which allow intercommunication intervals to slowly grow and potentially be unbounded. We propose logarithmically growing upper bounds which guarantee convergence.

Distributed synchronous systems require coordination between the agents. Asynchronous systems, in contrast, do not depend on global clock signals. This can save power as agents do not have to perform computation and communication at every iteration. However, it might require more iterations to converge.
While existing works on push-sum in the presence of link failures assume synchronous \cite{Vaidya} or broadcast asymmetric \cite{Schenato} communication setting, our major contribution in this paper is to develop a \textit{fully asynchronous} robust push-sum algorithm that allows the successive link failures to grow to infinity.

The rest of the paper is organized as follows. In Section \ref{sec:notation formulation} we introduce our notation and define the problem. In Sections \ref{sec:Ordinary} and \ref{sec:push-sum} we study ordinary consensus and push-sum algorithms, respectively, and state our convergence results. In Section \ref{sec:ra_ac}, we propose an asynchronous push-sum algorithm which is robust to unreliable communication links, followed by concluding remarks in Section \ref{sec:conclusion}.

\section{Problem Formulation}\label{sec:notation formulation}
\subsection{Notations and Definitions}
Suppose $\textbf{A}$ is a matrix, by $A_{ij}$ we denote its $(i,j)$ entry.
A matrix is called \textit{(row) stochastic} if it is non-negative and the sum of the elements of each row equals to one. Similarly, a matrix is \textit{column stochastic} if its transpose is stochastic. A matrix is called \textit{doubly stochastic} if it is both column and row stochastic.

To a non-negative matrix $\textbf{A} \in \mathbb{R}^{n\times n}$ we associate a directed graph $\mathcal{G}_\textbf{A}$ with vertex set $\mathcal{N}=\{1,2,\ldots,n\}$ and edge set $\mathcal{E}_\textbf{A}=\{(i,j)\vert  A_{ji}>0 \}$. Note that the graph might contain self-loops. 

By $[\textbf{A}]_\alpha$ we denote the \textit{thresholded} matrix obtained by setting every element of $\textbf{A}$ smaller than $\alpha$ to zero.

Given a sequence of matrices $\textbf{A}^0,\textbf{A}^1,\textbf{A}^2,\ldots$, we denote by $\textbf{A}^{k_2:k_1},k_2\geq k_1$, the product of elements $k_1$ to $k_2$ of the sequence, inclusive, in the following order:
\begin{displaymath}
\textbf{A}^{k_2:k_1}=\textbf{A}^{k_2}\textbf{A}^{k_2-1}\cdots \textbf{A}^{k_1}.
\end{displaymath}

Node $i$ is an \textit{in-neighbor} of node $j$, if there is a directed link from $i$ to $j$. Hence $j$ would be an \textit{out-neighbor} of node $i$. We denote the set of in-neighbors and out-neighbors of node $i$ at time $k$ with $N_i^{-,k}$ and $N_i^{+,k}$, respectively. Moreover, we denote the number of in-neighbors and out-neighbors of node $i$ at time $k$ with $d_i^{-,k}$ and $d_i^{+,k}$, as its \textit{in-degree} and \textit{out-degree}, respectively. If the graph is fixed, we will simply drop the index $k$ in the aforementioned notations.

By $x_{\min}$ and $x_{\max}$ we denote $\min_i x_i$ and $\max_i x_i$, respectively, unless mentioned otherwise.
We also denote a $n \times 1$ column vector of all ones by $\textbf{1}_n$, or $\textbf{1}$ when its size is clear from the context.

We sometimes use the notion of \textit{mass} to denote the value an agent holds, sends or receives. With that in mind, we can think of a value being sent from one node, as a mass being transferred. 

\subsection{Problem Formulation}
Consider a set of $n$ agents $\mathcal{N} = \{1,2,\ldots,n\}$, where each agent  $i$ holds an initial scalar value $x_i^0$. These agents communicate with each other through a sequence of directed graphs. Our goal is to develop protocols through which these agents communicate and update their values so that they reach consensus. Throughout this paper we use the terms \textit{agents} and \textit{nodes} interchangeably. 

Ordinary consensus and push-sum are two main algorithms proposed for this purpose.
In ordinary consensus, each node updates its value by forming a convex combination of the values of its in-neighbors.
In push-sum, average consensus is reached by running two parallel iterations in which, each node splits and sends its value to its out-neighbors and updates its own value by forming the sum of the messages that it has received.

\section{Ordinary Consensus}\label{sec:Ordinary}
Although the main target of this paper is push-sum, in this section we state and prove similar results for ordinary consensus. Comparable results can be found in \cite{lorenz2011convergence}, however the proofs provided here are necessary to understand the methods used in the following sections.

Linear consensus is defined as,
\begin{equation}\label{eq:1}
\textbf{x}^{k+1} = \textbf{A}^k \textbf{x}^k,\;\;k=0,1,\ldots,
\end{equation}
where the matrices $\textbf{A}^k$ are stochastic and $\textbf{x}^k$ is constructed by collecting all $x_i^k$ in a column vector.
Under the following conditions, the iteration \eqref{eq:1} results in consensus, meaning all the $x^k_{i}$ converge to the same value as $k\rightarrow\infty$.

The following assumption ensures sufficient connectivity of the graphs.

\begin{assumption}\label{assump:B_k-connectivity}
There exist a sequence $b_1,b_2,\ldots$ of positive integers such that when we partition the sequence of graphs $\mathcal{G}^0,\mathcal{G}^1,\mathcal{G}^2,\ldots$ to consecutive blocks of length $b_k$, $k=1,2,\ldots$, the graph constructed by the union of the edges in each block, is strongly connected. Also each graph $\mathcal{G}^k$ has a self-loop at every node.
\end{assumption}
Let us define $\mu_0=\lambda_0=0$, and for $k\geq1$:
\begin{gather}
\mu_k=\sum_{j=1}^{k}b_j,\label{eq:mu}\\
\lambda_k=\sum_{j=(k-1)n+1}^{kn}b_j=\mu_{kn} - \mu_{(k-1)n}.\label{eq:lambda}
\end{gather}

The following proposition states sufficient conditions for the convergence of ordinary consensus with growing intercommunication intervals.
\begin{proposition}\label{pro:B_k-convergence}
Suppose there exist some $\alpha>0$ such that the sequence of graphs $\mathcal{G}_{[\textbf{A}^0]_\alpha},\mathcal{G}_{[\textbf{A}^1]_\alpha},\mathcal{G}_{[\textbf{A}^2]_\alpha},\ldots$ satisfies Assumption \ref{assump:B_k-connectivity}. If there exist some $K\geq1,T\geq0,$ such that $\lambda_k\leq-\frac{\ln(k+T)}{\ln(\alpha)}$ for all $k\geq K$, then $x^k$ converges to a limit in span\{\textbf{1}\}.
\end{proposition}

Before proving the proposition, we need the following lemmas and definitions.
Given a sequence of graphs $\mathcal{G}^0,\mathcal{G}^1,\mathcal{G}^2,\ldots$, we will say node $b$ is reachable from node $a$ in time period $k_1$ to $k_2$ ($k_1<k_2$), if there exists a sequence of directed edges $e^{k_1},e^{k_1+1},\ldots,e^{k_2}$ such that $e^k$ is in $\mathcal{G}^k$, destination of $e^k$ is the origin of $e^{k+1}$ for $k_1\leq k<k_2$, and the origin of $e^{k_1}$ is $a$ and the destination of $e^{k_2}$ is $b$.

\begin{lemma}\label{lem:strictly positive matrix}
Suppose there exists some $\alpha>0$ such that the sequence of graphs $\mathcal{G}_{[\textbf{A}^0]_\alpha},\mathcal{G}_{[\textbf{A}^1]_\alpha},\mathcal{G}_{[\textbf{A}^2]_\alpha},\ldots$ satisfies Assumption \ref{assump:B_k-connectivity}. Then for $l\geq 0$, $\textbf{A}^{\mu_{l+n}-1:\mu_l}$ is a strictly positive matrix, with its elements at least $\alpha^{\mu_{l+n}-\mu_l}$.
\end{lemma}

\begin{proof}
Consider the set of reachable nodes from node $i$ in time period $k_1$ to $k_2$ in the graph sequence $\mathcal{G}_{[\textbf{A}^0]_\alpha},\mathcal{G}_{[\textbf{A}^1]_\alpha},\mathcal{G}_{[\textbf{A}^2]_\alpha},\ldots,$ and denote it by $N^{k_2:k_1}$. Since by Assumption \ref{assump:B_k-connectivity} each of these graphs has self-loop at every node, the set of reachable nodes never decreases. If $N^{\mu_{l+m}-1:\mu_l}\neq\{1,2,\ldots,n\}$ then $N^{\mu_{l+m+1}-1:\mu_l}$ is a strict super-set of $N^{\mu_{l+m}-1:\mu_l}$; because in period $\mu_{l+m}$ to $\mu_{l+m+1}-1$ there is an edge in some $\mathcal{G}_{[\textbf{A}^i]_\alpha}$ leading from the set of reachable nodes from $i$, to those not reachable from $i$; this is true because the union of the graphs in block $\mu_{l+m}$ to $\mu_{l+m+1}-1$ is strongly connected. Hence we conclude $N^{\mu_{l+n}-1:\mu_l}=\{1,2,\ldots,n\}$ and $\textbf{A}^{\mu_{l+n}-1:\mu_l}$ is strictly positive. Furthermore, since every positive element of $[\textbf{A}^k]_\alpha$ is at least $\alpha$ by construction, every element of $\textbf{A}^{\mu_{l+n}-1:\mu_l}$ is at least $\alpha^{\mu_{l+n}-\mu_l}$.
\end{proof}

\begin{lemma}\label{lem:max-min inequality}
Suppose $\textbf{A}$ is a stochastic matrix with entries at least $\beta>0$. If $\textbf{v}=\textbf{A}\textbf{u}$ then,
\begin{equation}
v_{\max}-v_{\min}\leq (1-n\beta)\left(u_{\max}-u_{\min}\right).
\end{equation}
\end{lemma}
This lemma is proved in \cite[Theorem 3.1 \& Exercise 3.8]{seneta2006non}.

\begin{lemma}\label{lem:stochastic matrices}
Suppose $\textbf{A}$ is a stochastic matrix and $\textbf{v}=\textbf{A}\textbf{u}$. Then for all $i$,
\begin{equation}
u_{\min} \leq v_i \leq u_{\max}.
\end{equation}
\end{lemma}
This lemma holds true because each $v_i$ is a convex combination of elements of $\textbf{u}$.

\begin{lemma}\label{lem:alpha}
Suppose $0<\alpha_k<1$ for $k=1,\ldots,\infty$, then $\prod_{k=1}^{\infty}\left(1-\alpha_k\right)=0$ if and only if $\sum_{k=1}^{\infty}\alpha_k=\infty$.
\end{lemma}
This lemma is proved in \cite[Appendix: Theorem 1.9]{bremaud2013markov} and we will skip the proof here.

\begin{proof}[Proof of Proposition \ref{pro:B_k-convergence}]
By Lemma \ref{lem:strictly positive matrix}, we have for $k\geq 1$,
\begin{displaymath}
\left[\textbf{A}^{\mu_{kn}-1:\mu_{(k-1)n}}\right]_{ij}\geq\alpha^{\mu_{kn}-\mu_{(k-1)n}}=\alpha^{\lambda_{k}}.
\end{displaymath}
Applying Lemma \ref{lem:max-min inequality}, we get,
\begin{equation}\label{eq:max-min ordinary}
x_{\max}^{\mu_{kn}}-x_{\min}^{\mu_{kn}}\leq \left(1-n\alpha^{\lambda_{k}}\right)\left(x_{\max}^{\mu_{(k-1)n}}-x_{\min}^{\mu_{(k-1)n}}\right).
\end{equation}
Hence, using \eqref{eq:max-min ordinary} for $k=1,\ldots,l$ we obtain,
\begin{equation*}
x_{\max}^{\mu_{ln}}-x_{\min}^{\mu_{ln}}\leq \prod_{k=1}^l\left(1-n\alpha^{\lambda_{k}}\right)\left(x_{\max}^0-x_{\min}^0\right).
\end{equation*}
We have $0<\alpha<1$ and $\lambda_k\leq-\frac{\ln(k+T)}{\ln(\alpha)}$ for all $k\geq K$. It follows,
\begin{align*}
\sum_{k=1}^{\infty}n\alpha^{\lambda_k}\geq \sum_{k=K}^{\infty}n\alpha^{\lambda_k}\geq\sum_{k=K}^{\infty}n\alpha^{-\frac{\ln(k+T)}{\ln(\alpha)}}
=\sum_{k=K}^{\infty}n\left(\alpha^{\frac{1}{\ln(\alpha)}}\right)^{-\ln(k+T)}
=\sum_{k=K}^{\infty}\frac{n}{k+T}=\infty.
\end{align*}
Using Lemmas \ref{lem:stochastic matrices} and \ref{lem:alpha} and \eqref{eq:max-min ordinary} we conclude that Proposition \ref{pro:B_k-convergence} holds.
\end{proof} 
Proposition \ref{pro:B_k-convergence} proves the convergence of $x_i^k$'s to a value which is not necessarily the total average and depends on the sequence of matrices. However if the matrices $\textbf{A}^k$ are doubly stochastic, the sum of the values of all nodes (agents) is preserved and therefore the algorithm converges to \textit{average} consensus.

Slight modifications to Example 1.2, Chapter 7 of \cite{bertsekas1989parallel} shows that if intercommunication intervals grow logarithmically in time, ordinary consensus fails to reach consensus. 

\section{Push-Sum}\label{sec:push-sum}
Push-sum is an algorithm that reaches average consensus and does not require \textit{doubly} stochastic matrices, as opposed to ordinary average consensus. Here, we assume each node knows its out-degree at every iteration. Under this assumption, it turns out that average consensus is possible and may be accomplished using the following iteration,
\begin{align}\label{eq:push-sum}
x_{i}^{k+1}&=\sum_{j\in N_{i}^{-,k}}\frac{x_j^k}{d_j^{+,k}},\nonumber\\
y_{i}^{k+1}&=\sum_{j\in N_{i}^{-,k}}\frac{y_j^k}{d_j^{+,k}},\\
z_i^{k+1}&=\frac{x_i^{k+1}}{y_i^{k+1}},\nonumber
\end{align}
where the auxiliary variables $y_i$ are initialized as $y_i^0=1$ and are collected in a column vector $\textbf{y}$.
This iteration is implemented in a distributed way using two steps. First each node $i$ broadcasts $x_i^k/d_i^{+,k}$ to its out-neighbors. Next, every node sets $x_i^{k+1}$ to be the sum of the incoming messages. Variables $y_i^k$ follow the same evolution. $z_i^k$ may be thought of as node $i$'s estimation of the average.

We define $\textbf{W}^k$ to be the matrix such that iteration (\ref{eq:push-sum}) may be written as,
\begin{align}
\textbf{x}^{k+1}&=\textbf{W}^k\textbf{x}^k,\nonumber\\
\textbf{y}^{k+1}&=\textbf{W}^k\textbf{y}^k.\nonumber
\end{align}

Next, we will state and prove a proposition regarding the sufficient conditions for the push-sum algorithm to converge.

\begin{proposition}\label{pro:push-sum}
Suppose the sequence of graphs $\mathcal{G}_{\textbf{W}^0},\mathcal{G}_{\textbf{W}^1},\mathcal{G}_{\textbf{W}^2},\ldots,$ satisfies Assumption \ref{assump:B_k-connectivity}. If there exist some $K\geq1,T\geq0,$ such that $\lambda_k\leq\frac{\ln(k+T)}{2\ln(n)}$ for all $k\geq K$, by implementing the push-sum algorithm (\ref{eq:push-sum}), it follows
\begin{equation*}
\lim_{k\to\infty}z_i^k=\frac{\sum_{j=1}^{n}x_j^0}{n}.
\end{equation*}
\end{proposition}

Note that positive elements of $\textbf{W}^k$ are at least $1/d_{\max}^{+,k}\geq 1/n$. Moreover, $\textbf{W}^k$ is column stochastic, i.e.,
\begin{displaymath}
\textbf{1}^T\textbf{W}^k=\textbf{1}^T.
\end{displaymath}
consequently, the sum of $x^k$ and $y^k$ are preserved, i.e., 
\begin{align}
\sum_{i=1}^{n}x_i^k&=\sum_{i=1}^{n}x_i^0,\label{eq:sum x}\\
\sum_{i=1}^{n}y_i^k&=\sum_{i=1}^{n}y_i^0=n.\label{eq:sum y=n}
\end{align}

Before proving the proposition, we need the following lemma, which establishes bounds for $y_i^{\mu_{ln}}$.

\begin{lemma}\label{lemma:yk bound}
Suppose the Assumptions stated in Proposition \ref{pro:push-sum} are satisfied. The following bounds on $y_i^{\mu_{ln}}$ hold for any $l\geq 1$:
\begin{equation}\label{eq:yk bound}
\left(\frac{1}{n}\right)^{\lambda_l-1}\leq y_i^{\mu_{ln}}\leq n.
\end{equation}
\end{lemma}

\begin{proof}
We observe that for $l\geq 1$,
\begin{equation}\label{eq:4}
\textbf{y}^{\mu_{ln}}=\textbf{W}^{\mu_{ln}-1:0}\textbf{1}.
\end{equation}
By Lemma \ref{lem:strictly positive matrix}, the matrix $\textbf{W}^{\mu_{ln}-1:\mu_{(l-1)n}}$ is strictly positive with it's elements at least $(1/n)^{\lambda_l}$. Hence $\textbf{W}^{\mu_{ln}-1:0}$ is the product of a strictly positive column stochastic matrix and other column stochastic matrices; consequently each of its entries are at least $(1/n)^{\lambda_l}$. Using \eqref{eq:4} we derive the left part of \eqref{eq:yk bound}.

Since $y_j^k>0$ for all $j$ and $k$, using \eqref{eq:sum y=n}, the right part of \eqref{eq:yk bound} is concluded.
\end{proof}

Now we can proceed with the proof of Proposition \ref{pro:push-sum}.

\begin{proof}[Proof of Proposition \ref{pro:push-sum}]
We start by rewriting the evolution of $\textbf{z}^k$ in a matrix form. The method to accomplish this is based on an observation from \cite{seneta2006non}. Using \eqref{eq:push-sum}, we have $x_i^k=z_i^ky_i^k$ and therefore,
\begin{equation}
z_i^{k+1}y_i^{k+1}=\sum_{j=1}^{n} W^k_{ij}z_j^ky_j^k,\nonumber
\end{equation}
or
\begin{equation}\label{eq:z=ywy}
z_i^{k+1}=\sum_{j=1}^n\left(y_i^{k+1}\right)^{-1}W^k_{ij}z_j^ky_j^k,
\end{equation}
where in the last step we used the fact that $y_i^k\neq0$, which is true by Lemma \ref{lemma:yk bound}. Define,
\begin{equation}\label{eq:Pk definition}
\textbf{P}^k=\left(\textbf{Y}^{k+1}\right)^{-1}\textbf{W}^k\textbf{Y}^k,
\end{equation}
where $\textbf{Y}^k={\rm diag}\left(\textbf{y}^k\right)$.
Using \eqref{eq:z=ywy} we have,
\begin{equation*}
\textbf{z}^{k+1}=\textbf{P}^k\textbf{z}^k.
\end{equation*}
Moreover, $\textbf{P}^k$ is stochastic:
\begin{align*}
\textbf{\textbf{P}}^k\textbf{1}&=\left(\textbf{Y}^{k+1}\right)^{-1}\textbf{W}^k\textbf{Y}^k\textbf{1}=\left(\textbf{Y}^{k+1}\right)^{-1}\textbf{W}^k\textbf{y}^k \\
&=\left(\textbf{Y}^{k+1}\right)^{-1}\textbf{y}^{k+1}{}={}\textbf{1}.
\end{align*}
Using \eqref{eq:Pk definition}, we obtain
\begin{equation}\label{eq:P=YWY}
\textbf{P}^{\mu_{kn}-1:\mu_{(k-1)n}}=\left(\textbf{Y}^{\mu_{kn}}\right)^{-1}\textbf{W}^{\mu_{kn}-1:\mu_{(k-1)n}}\textbf{Y}^{\mu_{(k-1)n}}.
\end{equation}
By Lemma \ref{lem:strictly positive matrix} the matrix $\textbf{W}^{\mu_{kn}-1:\mu_{(k-1)n}}$ is strictly positive; therefore using \eqref{eq:yk bound} and \eqref{eq:P=YWY} , $\textbf{P}^{\mu_{kn}-1:\mu_{(k-1)n}}$ is a strictly positive matrix with its elements at least
\begin{equation*}
\alpha_k=\frac{1}{n}\left(\frac{1}{n}\right)^{\lambda_k}\left(\frac{1}{n}\right)^{\lambda_{k-1}-1}=\left(\frac{1}{n}\right)^{\lambda_{k}+\lambda_{k-1}}.
\end{equation*}
Using Lemma \ref{lem:max-min inequality} we obtain,
\begin{equation}
z_{\max}^{\mu_{kn}}-z_{\min}^{\mu_{kn}}\leq (1-n\alpha _k)\left(z_{\max}^{\mu_{(k-1)n}}-z_{\min}^{\mu_{(k-1)n}}\right),\nonumber
\end{equation}
and consequently,
\begin{equation}\label{eq:max-min push-sum}
z_{\max}^{\mu_{ln}}-z_{\min}^{\mu_{ln}}\leq \prod_{k=1}^l(1-n\alpha _k)\left(z_{\max}^{0}-z_{\min}^{0}\right).
\end{equation}
Moreover,
\begin{align*}
\sum_{k=1}^{\infty}n\alpha_k \geq \sum_{k=K}^{\infty}n\alpha_k &=  \sum_{k=K}^{\infty}n\left(\frac{1}{n}\right)^{\lambda_{k}+\lambda_{k-1}}\\
&\geq \sum_{k=K}^{\infty}n\left(\frac{1}{n}\right)^{\frac{\ln(k+T)}{2\ln(n)}+\frac{\ln(k-1+T)}{2\ln(n)}}\\
&\geq \sum_{k=K}^{\infty}n\left(\frac{1}{n}\right)^{\frac{\ln(k+T)}{\ln(n)}}\\
&= \sum_{k=K}^{\infty}\frac{n}{k+T}=\infty.
\end{align*}
Hence using Lemma \ref{lem:alpha} and \eqref{eq:max-min push-sum}, $z_{\max}^{\mu_{ln}}-z_{\min}^{\mu_{ln}}$ converges to zero as $l\to \infty$. By Lemma \ref{lem:stochastic matrices} we conclude that $\lim_{k\to \infty}z_i^k$ exists and we denote it by $z_{\infty}$.
We have,
\begin{align*}
z_\infty &= z_\infty \lim_{k\to \infty} \left(\frac{\sum_{i=1}^n y_i^k}{\sum_{i=1}^n y_i^k}\right)\\
&= \lim_{k\to \infty}\left(\frac{\sum_{i=1}^n z_i^ky_i^k}{n}+\frac{\sum_{i=1}^n (z_\infty -z_i^k)y_i^k}{n}\right)\nonumber\\
&= \frac{\sum_{i=1}^n x_i^k	}{n}+\lim_{k\to \infty}\left(\frac{\sum_{i=1}^n (z_\infty -z_i^k)y_i^k}{n}\right)\\
&= \frac{\sum_{i=1}^n x_i^0}{n},
\end{align*}
where the last equality holds due to the sum preservation property, \eqref{eq:sum x}.
\end{proof}

\section{Robust Asynchronous Push-Sum}\label{sec:ra_ac}
Here we describe and study another algorithm for average consensus, in which the communication system is asynchronous and unreliable. In an unreliable setting, communication links might fail to transmit data packets and information might get lost.

This algorithm is originally inspired by the algorithm  proposed by \cite{Vaidya}, but under asynchronous communication. As the algorithm in \cite{Vaidya}, this algorithm is also based on the push-sum consensus. \cite{Schenato} has proved exponential convergence of this algorithm for the case when at each iteration only one node wakes up and transmits. Here we modify the algorithm presented by \cite{Schenato} and show that average consensus still holds while allowing for any subset of nodes to perform updates at each iteration.

In this algorithm, as opposed to the previous ones, we assume nodes do not have self-loops.

\begin{algorithm}
\caption{Robust Asynchronous Push-Sum}
\begin{algorithmic}[1]
\STATE Initialize the algorithm with $\textbf{y}^0=\textbf{1}$, $\sigma_i^{x,0}=\sigma_i^{y,0}=0$ $\forall i\in\{1,\ldots,n\}$ and $\rho_{ji}^{x,0}=\rho_{ji}^{y,0}=0$, $\forall (i,j)\in \mathcal{E}$.
\STATE At every iteration $k$, for every node $i$:
\IF{node $i$ wakes up}
\STATE $\sigma_i^{x,k+1}=\sigma_i^{x,k}+\frac{x_i^k}{d_i^{+}+1}$;
\STATE $\sigma_i^{y,k+1}=\sigma_i^{y,k}+\frac{y_i^k}{d_i^{+}+1}$;
\STATE $x_i^{k+1}=\frac{x_i^k}{d_i^{+}+1}$;
\STATE $y_i^{k+1}=\frac{y_i^k}{d_i^{+}+1}$;
\STATE Node $i$ broadcasts $\sigma_i^{x,k+1}$ and $\sigma_i^{y,k+1}$ to its out-neighbors: $N_i^+$
\ENDIF
\IF{node $i$ receives $\sigma_j^{x,k+1}$ and $\sigma_j^{y,k+1}$ from $j\in N_i^-$}
\STATE $\rho_{ij}^{x,k+1}=\sigma_j^{x,k+1}$;
\STATE $\rho_{ij}^{y,k+1}=\sigma_j^{y,k+1}$;
\STATE $x_i^{k+1}=x_i^{k+1}+\rho_{ij}^{x,k+1}-\rho_{ij}^{x,k}$;
\STATE $y_i^{k+1}=y_i^{k+1}+\rho_{ij}^{y,k+1}-\rho_{ij}^{y,k}$;
\ENDIF
\STATE Other variables remain unchanged.

\end{algorithmic}
\end{algorithm}

The impressive idea proposed by \cite{Vaidya} that allows us to overcome the issue of unreliable of links, is that of introducing the counters: in particular each node $i$ has a counter $\sigma_i^{x,k}$ ($\sigma_i^{y,k}$ respectively) to keep track of the total $x$-mass ($y$-mass) sent by itself to its neighbors from time 0 to time $k$, and counters $\rho_{ij}^{x,k}$ ($\rho_{ij}^{y,k}$ respectively) $\forall j\in N_i^-$, to take into account the total $x$-mass ($y$-mass) received from its neighbor $j$ from time 0 to time $k$.

While in reality, nodes will perform computations when they wake up; to make the analysis easier, we assume nodes perform computations (but no transmission) when they are not awake.

Next, we state and prove the main theorem of this paper, which shows that the algorithm above reaches average consensus under sufficient connectivity assumptions.

\begin{theorem}\label{theorem:robust push-sum}
Suppose we apply the Robust Asynchronous Push-Sum algorithm to a set of agents communicating with each other through a strongly connected graph $\mathcal{G=(N,E)}$, where $\mathcal{E}$ does not have self-loops. Let $\mathcal{G}^0,\mathcal{G}^1,\ldots,$ be the sequence of graphs $\mathcal{G}^i=(\mathcal{N},\mathcal{E}^i)$, $\mathcal{E}^i\subset\mathcal{E}$, containing only the links which transmit successfully at iteration $i$. Also, suppose there is another sequence $b_1,b_2,\ldots,$ of positive integers such that, if we split the sequence of $\mathcal{G}^0,\mathcal{G}^1,\ldots,$ to consecutive blocks of length $b_i$, the union of graphs of each block is equal to $\mathcal{G}$; i.e., $\cup_{i=\mu_k}^{\mu_{k+1}-1}\mathcal{E}^i=\mathcal{E}, \forall k\geq 0$, 
where $\mu_k$ and $\lambda_k$ are defined in (\ref{eq:mu}) and (\ref{eq:lambda}). Suppose that there exists some $K\geq1,T\geq0,$ such that $\lambda_k\leq \frac{\ln(k+T)}{6\ln(n)}$, $\forall k\geq K$.
Then, $z_i^k=x_i^k/y_i^k$ converges to the average of $\textbf{x}^0$, i.e.,
\begin{equation}
\lim_{k\to\infty}z_i^k=\frac{\sum_{j=1}^{n}x_j^0}{n}.\nonumber
\end{equation}
\end{theorem}

\begin{proof}
Similar to the proofs of the previous propositions, here we first rewrite the evolution of $\textbf{x}^k$ and $\textbf{y}^k$ in a matrix form. We show these matrices are column stochastic. Then we write the evolution of the agents' estimate of the average, $\textbf{z}^k$, in matrix form. Finally, we exploit the properties of these matrices to show the convergence of $z_i^k$ to one limit which turns out to be the average.

Before we rewrite the iteration in a matrix form, we introduce the indicator variables $\tau_i^k$, for $i=1,2,\ldots,n$, and $\tau_{ij}^k$, for $(i,j)\in \mathcal{E}$. $\tau_i^k$ is equal to 1 if node $i$ wakes up at time $k$, and is 0 otherwise. Likewise $\tau_{ij}^k$ is 1 whenever node $i$ wakes up at time $k$, $j\in N_i^{+}$ and the edge $(i,j)$ is reliable, while it is 0 otherwise.

Let us introduce the following variables:
\begin{align*}
u_{ij}^k&=\sigma_i^{x,k}-\rho_{ji}^{x,k},\qquad \forall(i,j)\in \mathcal{E}, \\
v_{ij}^k&=\sigma_i^{y,k}-\rho_{ji}^{y,k},\qquad \forall(i,j)\in \mathcal{E},
\end{align*}
which are, intuitively, the total $x$-mass and $y$-mass, respectively, that has been sent by node $i$ but due to link failures has not been delivered to node $j$ yet.
The evolution of $y$-mass is exactly the same as $x$-mass; hence to avoid repetition, we only analyze the evolution of $\textbf{x}^k$ and $u_{ij}^k$.
We can write the update equations:
\begin{gather}
u_{ij}^{k+1}=\left(1-\tau_i^k\tau_{ij}^k\right)\left(u_{ij}^k+\tau_i^k\frac{x_i^k}{d_i^{+}+1}\right),\label{eq:update1}\\
x_i^{k+1} = \sum_{j \in N_i^-}\left(\frac{x_j^k}{d_j^{+}+1}+u_{ji}^k\right)\tau_j^k\tau_{ji}^k
+ x_i^k\left(1-\tau_i^k+\frac{\tau_i^k}{d_i^{+}+1}\right).\label{eq:update2}
\end{gather}

Let us introduce the column vectors $\textbf{u}^k$ and $\textbf{v}^k$ which collect all different $u_{ij}^k$ and $v_{ij}^k$, respectively. Moreover, let us introduce the column vectors $\bm{\phi}^{(x)}(k)=\left[ (\textbf{x}^k)^T,(\textbf{u}^k)^T\right]^T$, $\bm{\phi}^{(y)}(k)=\left[ (\textbf{y}^k)^T,(\textbf{v}^k)^T\right]^T\in \mathbb{R}^{n+m}$, where $m=\vert \mathcal{E}\vert$. Using \eqref{eq:update1} and \eqref{eq:update2} we can rewrite the algorithm in the following matrix form:
\begin{align}
\bm{\phi}^{(x)}(k+1)=\textbf{M}^k\bm{\phi}^{(x)}(k),\\
\bm{\phi}^{(y)}(k+1)=\textbf{M}^k\bm{\phi}^{(y)}(k).
\end{align}
\begin{lemma}\label{lem: M col stoc}
$\textbf{M}$ is column stochastic and each positive element of it is at least $1/(\max_i \{d_i^+\}+1)$. Also we have for $1\leq i \leq n$:
\begin{equation}
M_{ii}^k=\begin{cases}
1,&\text{if }\tau_i^k=0,\\
\frac{1}{d_i^++1},&\text{if }\tau_i^k=1.
\end{cases}
\end{equation}
\end{lemma}

\begin{proof}
Let us first consider the $i^{th}$ column of $\textbf{M}^k$, with $1\leq i \leq n$. The element $M_{ii}^k$ indicates how $x_i^k$ influences $x_i^{k+1}$. Using \eqref{eq:update2}, it follows:
\begin{equation}\label{eq:M_ii}
M_{ii}^k=1-\tau_i^k+\frac{\tau_i^k}{d_i^{+}+1}=\begin{cases}
1,&\text{if }\tau_i^k=0,\\
\frac{1}{d_i^++1},&\text{if }\tau_i^k=1.
\end{cases} 
\end{equation}
The element $M_{ji}^k$, $j \in \{1,\ldots,n\}\setminus\{i\}$ indicates how $x_i^k$ influences $x_j^{k+1}$. It holds,
\begin{equation}\label{eq:M_ji}
M_{ji}^k=\begin{cases}
\frac{\tau_i^k\tau_{ij}^k}{d_i^{+}+1},&\text{if }j \in N_i^+,\\
0,&\text{otherwise.}
\end{cases}
\end{equation}
Finally, if $h\in \{n+1,\ldots,n+m\}$ is such that $\bm{\phi}^{(x)}_h(k)=u_{rj}^k$; the element $M_{hi}^k$ indicates how $x_i^k$ influences $u_{rj}^{k+1}$, we have
\begin{equation}\label{eq:M_li}
M_{hi}^k=\begin{cases}
\frac{(1-\tau_{ij}^k)\tau_i^k}{d_i^{+}+1},&\text{if }r=i,\\
0,&\text{otherwise.}
\end{cases}
\end{equation}
Using \eqref{eq:M_ii}-\eqref{eq:M_li}, entries of $i^{th}$ column of $\textbf{M}^k$ sum to 1.

Now we consider the $h^{th}$ column of $\textbf{M}^k$, $h\in \{n+1,\ldots,n+m\}$. Suppose $\bm{\phi}^{(x)}_h(k)=u_{ij}^k$, we have
\begin{align}
M_{jh}^k&=\tau_i^k\tau_{ij}^k,\label{eq:M_jh}\\ 
M_{hh}^k&=1-\tau_i^k\tau_{ij}^k,\label{eq:M_hh}
\end{align}
and all the other elements of $h^{th}$ column are zero. Using \eqref{eq:M_jh} and \eqref{eq:M_hh}, the entries of the $h^{th}$ column sum to 1 and hence the matrix $\textbf{M}^k$ is column stochastic.
\end{proof}

Let us augment the graph $\mathcal{G}^k$ to $\mathcal{H}^k=\mathcal{G}_{\textbf{M}^k}$ by adding auxiliary nodes $b_{ij}$, $\forall(i,j)\in\mathcal{E}$. Note that by Lemma \ref{lem: M col stoc}, node $i\in\{1,\ldots,n\}$ has self-loop all the time and node $b_{ij}$ has self-loop unless the link $(i,j)$ transmits reliably. Let us call nodes $b_{ij}$ \textit{buffers} and assign values $u_{ij}^k$ and $v_{ij}^k$ to them.

The algorithm is equivalent to the following process: Suppose node $i$ wakes up. If the link $(i,j)$ works properly, node $i$ sends some mass ($x_i^k/(d_i^++1)$ and $y_i^k/(d_i^++1)$) to node $j$ and also node $b_{ij}$ sends all of its mass ($u_{ij}^k$ and $v_{ij}^k$) to node $j$ and becomes zero. Otherwise, the mass is sent from node $i$ to node $b_{ij}$ instead of $j$. Then all the mass gets accumulated at node $b_{ij}$ because of its self-loop, until the link $(i,j)$ transmits reliably.

\begin{lemma}\label{lemma:positive n-rows}
The first n rows of $\textbf{M}^{\mu_{l+n}-1:\mu_l}$ are strictly positive, $l\geq 0$. The positive elements of this matrix are at least $\left(1/n\right)^{\mu_{l+n}-\mu_l}$.
\end{lemma}

\begin{proof}
Observing $\mathcal{H}^k$, every node $j\in \{1,\ldots,n\}$ has self-loop in every iteration and buffer $b_{ij}$ has self-loop unless link $(i,j)$ transmits successfully. We also know that during period $\mu_k$ to $\mu_{k+1}-1$, $k=0,1,\ldots$, each edge $(i,j)\in \mathcal{E}$ transmits successfully at least once. Moreover, $\mathcal{G}$ is strongly connected; Hence at the end of period $\mu_l$ to $\mu_{l+n}-1$, every node $j\in \{1,\ldots,n\}$ is reachable from all the nodes in graph $\mathcal{H}$. Also, since each positive element of $\textbf{M}^k$ is at least $1/n$, each positive element of $\textbf{M}^{\mu_{l+n}-1:\mu_l}$ is at least $(1/n)^{\mu_{l+n}-\mu_l}$.
\end{proof}

Define $\textbf{W}^k=\textbf{M}^{\mu_{(k+1)n}-1:\mu_{kn}}, k\geq 0$, which has positive elements of at least $\alpha^{\lambda_{k+1}}$ where $\alpha=1/n$. Then we have:
\begin{gather}
\begin{bmatrix}
\textbf{x}^{\mu_{(k+1)n}}\\ \textbf{u}^{\mu_{(k+1)n}}
\end{bmatrix}=\textbf{W}^k
\begin{bmatrix}
\textbf{x}^{\mu_{kn}}\\ \textbf{u}^{\mu_{kn}}
\end{bmatrix},\label{eq:xu=wxu} \\
\begin{bmatrix}
\textbf{y}^{\mu_{(k+1)n}}\\ \textbf{v}^{\mu_{(k+1)n}}
\end{bmatrix}=\textbf{W}^k
\begin{bmatrix}
\textbf{y}^{\mu_{kn}}\\ \textbf{v}^{\mu_{kn}}
\end{bmatrix}.\label{eq:yv=wyv} 
\end{gather}

Let us split the matrix $\textbf{W}^k$ to four sub-matrices as follows:
\begin{equation}\label{eq:W=ABCD}
\textbf{W}^k=\begin{bmatrix}
\textbf{A}^k&\textbf{B}^k\\\textbf{C}^k&\textbf{D}^k
\end{bmatrix},
\end{equation}
where $\textbf{A}^k\in\mathbb{R}^{n\times n}$, $\textbf{B}^k\in\mathbb{R}^{n\times m}$, $\textbf{C}^k\in\mathbb{R}^{m \times n}$ and $\textbf{D}^k\in\mathbb{R}^{m \times m}$. By Lemma \ref{lemma:positive n-rows} we know that matrices $\textbf{A}^k$ and $\textbf{B}^k$ are strictly positive.

For $h=1,\ldots,m$ define $r_h^k$ as follows:
\begin{equation*}
r_h^k=\begin{cases}
\frac{u_h^k}{v_h^k},&\text{if }v_h^k\neq0,\\
0,&\text{if }v_h^k=0.
\end{cases}
\end{equation*}
\begin{lemma}
$u_{ij}^k=0$ whenever $v_{ij}^k=0$.
\end{lemma}

\begin{proof}
Since $\textbf{v}^0=\textbf{0}_m$ and $\textbf{y}^0=\textbf{1}_n$ and node $i$ has self loop in graph $\mathcal{H}^k$ for all $k\geq 0$, $y_i^k$ is always positive. If $v_{ij}^k=0$, the last time the node $i$ has woken up, the link $(i,j)$ has worked successfully, or $i$ has not woken up yet. In either case, node $b_{ij}$ has no remaining ($x$ and $y$) mass and $u_{ij}^k=0$ holds.
\end{proof}

Therefore, the following always holds for $h=1,\ldots,m$:
\begin{equation}\label{eq:u=rv}
u_h^k=r_h^kv_h^k,
\end{equation}

Define $\bar{\textbf{x}}^k=\textbf{x}^{\mu_{kn}}$, $\bar{\textbf{y}}^k=\textbf{y}^{\mu_{kn}}$, $\bar{\textbf{u}}^k=\textbf{u}^{\mu_{kn}}$, $\bar{\textbf{v}}^k=\textbf{v}^{\mu_{kn}}$, $\bar{\textbf{z}}^k=\textbf{z}^{\mu_{kn}}$ and $\bar{\textbf{r}}^k=\textbf{r}^{\mu_{kn}}$. Using (\ref{eq:xu=wxu}) and (\ref{eq:W=ABCD}) we obtain:
\begin{align*}
\bar{z}_i^{k+1}\bar{y}_i^{k+1}=\bar{x}_i^{k+1} &=\sum_{j=1}^n A_{ij}^k\bar{x}_j^k + \sum_{j=1}^{m}B_{ij}^k\bar{u}_j^k\\
&=\sum_{j=1}^n A_{ij}^k\bar{z}_j^k\bar{y}_j^k + \sum_{j=1}^{m} B_{ij}^k\bar{r}_j^k\bar{v}_j^k.\nonumber
\end{align*}
Hence,
\begin{align*}
\bar{z}_i^{k+1}&= \left(\bar{y}_i^{k+1}\right)^{-1}\sum_{j=1}^n A_{ij}^k\bar{z}_j^k\bar{y}_j^k + \left(\bar{y}_i^{k+1}\right)^{-1}\sum_{j=1}^{m} B_{ij}\bar{r}_j^k\bar{v}_j^k,\\
\bar{\textbf{z}}^{k+1}&=\left(\textbf{Y}^{k+1}\right)^{-1}\textbf{A}^k\textbf{Y}^k\bar{\textbf{z}}^k+
\left(\textbf{Y}^{k+1}\right)^{-1}\textbf{B}^k\textbf{V}^k\bar{\textbf{r}}^k,
\end{align*}
where $\textbf{Y}^k={\rm diag}\left(\bar{\textbf{y}}^{k}\right)$ and $\textbf{V}^k={\rm diag}\left(\bar{\textbf{v}}^{k}\right)$. Note that $\bar{\textbf{y}}^k$ is strictly positive. Similarly, using \eqref{eq:yv=wyv}-\eqref{eq:u=rv} we have,
\begin{align*}
\bar{r}_i^{k+1}\bar{v}_i^{k+1}=\bar{u}_i^{k+1}&=\sum_{j=1}^n C_{ij}^k\bar{x}_j^k + \sum_{j=1}^{m} D_{ij}^k\bar{u}_j^k\nonumber\\
&=\sum_{j=1}^n C_{ij}^k\bar{z}_j^k\bar{y}_j^k + \sum_{j=1}^{m} D_{ij}^k\bar{r}_j^k\bar{v}_j^k.
\end{align*}
Here $\bar{\textbf{v}}^k$, as opposed to $\bar{\textbf{y}}^k$, is not necessarily strictly positive. Therefore instead of $\left(\textbf{V}^k\right)^{-1}$, we define the following:
\begin{equation*}
\tilde{v}_i^k=\begin{cases}
\frac{1}{\bar{v}_i^k},&\text{if }\bar{v}_i^k\neq0,\\
0,&\text{if }\bar{v}_i^k=0.
\end{cases}
\end{equation*}
It follows:
\begin{align*}
\bar{r}_i^{k+1} &= \tilde{v}_i^{k+1}\sum_{j=1}^n C_{ij}^k\bar{z}_j^k\bar{y}_j^k + \tilde{v}_i^{k+1}\sum_{j=1}^{m} D_{ij}^k\bar{r}_j^k\bar{v}_j^k,\\
\bar{\textbf{r}}^{k+1}&=\tilde{\textbf{V}}^{k+1}\textbf{C}^k\textbf{Y}^k\bar{\textbf{z}}^k+
\tilde{\textbf{V}}^{k+1}\textbf{D}^k\textbf{V}^k\bar{\textbf{r}}^k.
\end{align*}
where $\tilde{\textbf{V}}^k={\rm diag}(\tilde{v}^k)$. Thus,
\begin{equation}\label{eq:zr=pzr}
\begin{bmatrix}
\bar{\textbf{z}}^{k+1}\\\bar{\textbf{r}}^{k+1}
\end{bmatrix}
=\textbf{P}^k
\begin{bmatrix}
\bar{\textbf{z}}^{k}\\\bar{\textbf{r}}^{k}
\end{bmatrix},
\end{equation}
where,
\begin{equation}\label{eq:P=YAY YBV}
\textbf{P}^k=\begin{bmatrix}
\left(\textbf{Y}^{k+1}\right)^{-1}\textbf{A}^k\textbf{Y}^k & \left(\textbf{Y}^{k+1}\right)^{-1}\textbf{B}^k\textbf{V}^k\\
\tilde{\textbf{V}}^{k+1}\textbf{C}^k\textbf{Y}^k&\tilde{\textbf{V}}^{k+1}\textbf{D}^k\textbf{V}^k
\end{bmatrix}.
\end{equation}
Now we show that the sum of the elements of each row $1$ to $n$ of $\textbf{P}^k$ is equal to 1, but for the rest of the rows they either sum to 1 or they are all zeros.
\begin{align*}
\textbf{P}^k\begin{bmatrix}
\mathbf{1}_n\\
\bold{1}_m
\end{bmatrix}=\begin{bmatrix}
\left(\textbf{Y}^{k+1}\right)^{-1}\left(\textbf{A}^k\bar{\textbf{y}}^k+\textbf{B}^k\bar{\textbf{v}}^k\right)\\
\tilde{\textbf{V}}^{k+1}\left(\textbf{C}^k\bar{\textbf{y}}^k+\textbf{D}^k\bar{\textbf{v}}^k\right)
\end{bmatrix}
=\begin{bmatrix}
\left(\textbf{Y}^{k+1}\right)^{-1}\bar{\textbf{y}}^{k+1}\\
\tilde{\textbf{V}}^{k+1}\bar{\textbf{v}}^{k+1}
\end{bmatrix}=\begin{bmatrix}
\bold{1}_n\\
\text{1 or 0}\\
\vdots\\
\text{1 or 0}
\end{bmatrix}.
\end{align*}
The $(n+h)^{th}$ row of $\textbf{P}^k$ is zero if and only if $v_h^{k+1}$ is zero.

\begin{lemma}
For $k\geq 0$ and $1\leq i\leq n$ we have:
\begin{equation}
\alpha^{\lambda_k} \leq \bar{y}_i^k \leq n. 
\end{equation}
Moreover, for $1\leq h\leq m$ and $k\geq 1$ we have either $\bar{v}_h^k=0$ or,
\begin{equation}
\alpha^{\lambda_k+\lambda_{k-1}}\leq \bar{v}_h^k \leq n.
\end{equation}
\end{lemma}

\begin{proof}
We have for $k\geq 1$,
\begin{equation}
\begin{bmatrix}
\bar{\textbf{y}}^k\\ \bar{\textbf{v}}^k
\end{bmatrix}=\textbf{W}^{k-1:0}
\begin{bmatrix}
\bold{1}_n\\\bold{0}_m
\end{bmatrix},\nonumber
\end{equation}
where $\textbf{W}^{k-1:0}$ is the product of $\textbf{W}^{k-1}$ and other column stochastic matrices. By Lemma \ref{lemma:positive n-rows}, $\textbf{W}^{k-1}$ has positive first $n$ rows and its positive entries are at least $\alpha^{\lambda_k}$. Hence $\textbf{W}^{k-1:0}$ has positive first $n$ rows and its positive elements are at least $\alpha^{\lambda_k}$. We obtain for $1\leq i\leq n$,
\begin{equation*}
\bar{y}_i^k\geq\alpha^{\lambda_k},\textit{ for } k\geq 1.
\end{equation*}
Also since $\lambda_0=0$, $\bar{y}_i^0=1=\alpha^{\lambda_0}$.

Suppose node $h$ is the buffer of link $(i,j)$. If $\bar{v}_h^k$ is positive for some $k\geq 0$, it is because the last time node $i$ has woken up, link $(i,j)$ has failed and node $i$ has sent some value to $h$. Hence $W_{hi}^{k-1}\geq \alpha^{\lambda_k}$, and it follows,
\begin{equation*}
\bar{v}_h^k\geq\alpha^{\lambda_k}\bar{y}_i^{k-1}\geq \alpha^{\lambda_k+\lambda_{k-1}}.
\end{equation*}
Also, due to some preservation property, we have $\bar{y}_i^k, \bar{v}_h^k\leq n$, for all $i, h$ and $k$.
\end{proof}

Now we are able to find a lower bound on positive elements of $\textbf{P}^k$.
Let us divide $\textbf{P}^k$ to four sub-matrices as:
\begin{equation*}
\textbf{P}^k=\begin{bmatrix}
\textbf{E}^k&\textbf{F}^k\\\textbf{G}^k&\textbf{H}^k
\end{bmatrix},
\end{equation*}
where $\textbf{E}^k\in\mathbb{R}^{n\times n}$, $\textbf{F}^k\in\mathbb{R}^{n\times m}$, $\textbf{G}^k\in\mathbb{R}^{m \times n}$ and $\textbf{H}^k\in\mathbb{R}^{m \times m}$ are defined as in (\ref{eq:P=YAY YBV}).\\

By construction, positive elements of $\textbf{E}^k$ and $\textbf{G}^k$ are at least $\frac{1}{n}\alpha^{\lambda_{k+1}}\alpha^{\lambda_{k}}=\alpha^{\lambda_{k+1}+\lambda_k+1}$. Similarly, positive elements of $\textbf{F}^k$ and $\textbf{H}^k$ are at least $\alpha^{\lambda_{k+1}+\lambda_k+\lambda_{k-1}+1}$. Hence we can define the following lower bound for all positive elements of $\textbf{P}^k$:
\begin{equation}\label{eq:Pk lower bound}
\beta_k=\alpha^{\lambda_{k+1}+\lambda_k+\lambda_{k-1}+1}.
\end{equation}

We note the following facts by observing \eqref{eq:P=YAY YBV}:\\
$\bullet$ $\textbf{E}^k$ is strictly positive.\\
$\bullet$ if $\bar{v}_h^k$ is positive, the $h^{th}$ column of $\textbf{F}^k$ is strictly positive. Otherwise the whole $(n+h)^{th}$ column of $\textbf{P}^k$ is zero.\\
$\bullet$ if $\bar{v}_h^{k+1}$ is positive, the $h^{th}$ row of $\textbf{G}^k$ has at least one positive entry. This is true because during the time $\mu_{kn}$ to $\mu_{(k+1)n}-1$, the corresponding link $(i,j)$, transmits successfully at least once, which sets the values of $\bar{v}_h$ and $\bar{u}_h$ to 0. Therefore since $\bar{v}_h^{k+1}$ is positive, link $(i,j)$ has failed at least once after the last successful transmission. Hence, $C_{hi}^k$ is positive, and therefore $G_{hi}^k$ is also positive.

Define the index set $I^k=\{h\vert \bar{v}_h^k>0\}$. If $h\notin I^k$ we have $\bar{r}_h^k = \bar{v}_h^k = 0$, and also the $(n+h)^{th}$ column of $\textbf{P}^k$ has only zero entries; hence, $\bar{r}_h^k$ does not influence any variable of time $k+1$. We also have for $h\notin I^{k+1}$ the $(n+h)^{th}$ row of $\textbf{P}^k$ has only zero entries. Thus, $\bar{r}_h^{k+1}$ is formed by the sum of zero numbers. Intuitively, this means that for $h\notin I^k$, $\bar{r}_h^k$ is zero and so are the coefficients related to it in \eqref{eq:zr=pzr}. Therefore it gives us no meaningful information and it can be ignored. For the rest of the proof, we assume that all the variables $\bar{r}_h^k$ considered in the equations are the ones with $h\in I^k$.

We obtain:
\begin{gather*}
\bar{r}_{\max}^{k+1}\leq \beta^k \bar{z}_{\max}^k + (1-\beta^k)\max\{\bar{z}_{\max}^k,\bar{r}_{\max}^k\},\\
\bar{z}_{\max}^{k+1}\leq \beta^k \min\{\bar{z}_{\min}^k,\bar{r}_{\min}^k\}
 + (1-\beta^k)\max\{\bar{z}_{\max}^k,\bar{r}_{\max}^k\}.
\end{gather*}
Then,
\begin{equation*}
\max\{\bar{z}_{\max}^{k+1},\bar{r}_{\max}^{k+1}\}\leq \beta^k \bar{z}_{\max}^k + (1-\beta^k)\max\{\bar{z}_{\max}^k,\bar{r}_{\max}^k\}.
\end{equation*}
Similarly,
\begin{equation*}
\min\{\bar{z}_{\min}^{k+1},\bar{r}_{\min}^{k+1}\}\geq \beta^k \bar{z}_{\min}^k + (1-\beta^k)\min\{\bar{z}_{\min}^k,\bar{r}_{\min}^k\}.
\end{equation*}
We also have:
\begin{gather*}
\bar{z}_{\max}^{k+1}\leq \beta^k \sum_{i=1}^n \bar{z}_i^k + (1-n\beta^k)\max\{\bar{z}_{\max}^k,\bar{r}_{\max}^k\},\\
\bar{z}_{\min}^{k+1}\geq \beta^k \sum_{i=1}^n \bar{z}_i^k + (1-n\beta^k)\min\{\bar{z}_{\min}^k,\bar{r}_{\min}^k\}.
\end{gather*}
Thus,
\begin{equation*}
\bar{z}_{\max}^{k+1}-\bar{z}_{\min}^{k+1}\leq
(1-n\beta^k)\left(\max\{\bar{z}_{\max}^k,\bar{r}_{\max}^k\}-\min\{\bar{z}_{\min}^k,\bar{r}_{\min}^k\}\right).
\end{equation*}
Equivalently,
\begin{gather*}
s^{k+1}\leq \beta^k t^k + (1-\beta^k)s^k,\\
t^{k+1}\leq (1-n\beta^k)s^k,
\end{gather*}
where $s^k=\max\{\bar{z}_{\max}^k,\bar{r}_{\max}^k\}-\min\{\bar{z}_{\min}^k,\bar{r}_{\min}^k\}$ and $t^k=\bar{z}_{\max}^{k}-\bar{z}_{\min}^{k}$. Observing that $0\leq t^k\leq s^k$, we obtain:
\begin{align*}
s^{k+1}&\leq \beta^k(1-n\beta^{k-1})s^{k-1}+(1-\beta^k)s^k\\
&\leq \beta^k(1-n\beta^{k-1})s^{k-1}+(1-\beta^k)s^{k-1}\\
&= (1-n\beta^k \beta^{k-1})s^{k-1}.
\end{align*}
Hence $\lim_{k\to\infty}s^k=0$ if $\prod_{k=1}^{\infty}\left(1-n\beta^{2k}\beta^{2k-1}\right)=0$, which, by Lemma \ref{lem:alpha}, holds true if and only if  $\sum_{k=1}^{\infty}\beta^{2k}\beta^{2k-1}=\infty$. Using (\ref{eq:Pk lower bound}), we have:
\begin{align*}
\sum_{k=1}^{\infty}\beta^{2k}\beta^{2k-1}&=\sum_{k=1}^{\infty} \alpha^{\lambda_{2k+1}+2\lambda_{2k}+2\lambda_{2k-1}+\lambda_{2k-2}+2}\\
&\geq \frac{1}{n^2} \sum_{k=K}^{\infty} \alpha^{-\frac{\ln(2k+1+T)}{\ln(\alpha)}}\\
&= \frac{1}{n^2}\sum_{k=K}^{\infty}\frac{1}{2k+1+T} = \infty.
\end{align*}
Hence $\max\{\bar{z}_{\max}^k,\bar{r}_{\max}^k\}-\min\{\bar{z}_{\min}^k, \bar{r}_{\min}^k\}$ converges to 0 as $k$ goes to infinity. Combining this with Lemma \ref{lem:stochastic matrices} we obtain,
\begin{equation}
\lim_{k\to \infty}\bar{z}_i^k = \lim_{k\to \infty ,\text{ } h\in I^k }\bar{r}_h^k = L. \label{eq:lim rz}
\end{equation}
We have:
\begin{align*}
L &= L \lim_{k \to \infty} \frac{\sum_{i=1}^n \bar{y}_i^k + \sum_{h=1}^m \bar{v}_h^k}{\sum_{i=1}^n \bar{y}_i^k + \sum_{h=1}^m \bar{v}_h^k}\\
&= \lim_{k \to \infty}\left(\frac{\sum_{i=1}^n \bar{z}_i^k\bar{y}_i^k + \sum_{h=1}^m \bar{r}_h^k\bar{v}_h^k}{n}\right)+\lim_{k \to \infty}\left(\frac{\sum_{i=1}^n (L-\bar{z}_i^k)\bar{y}_i^k + \sum_{h=1}^m (L-\bar{r}_h^k)\bar{v}_h^k}{n}\right)\\
&= \lim_{k \to \infty}\left(\frac{\sum_{i=1}^n \bar{x}_i^k + \sum_{h=1}^m \bar{u}_h^k}{n}\right)
+\lim_{k \to \infty}\left(\frac{\sum_{i=1}^n (L-\bar{z}_i^k)\bar{y}_i^k + \sum_{h=1}^m (L-\bar{r}_h^k)\bar{v}_h^k}{n}\right)\\
&= \frac{\sum_{i=1}^n x_i^0}{n},
\end{align*}
where in the last equality, we used \eqref{eq:lim rz}, and the fact that  $\bar{v}_h^k=0$ for $h \notin I^k$.
\end{proof}

\section{Conclusion}\label{sec:conclusion}
In this paper we established sufficient conditions on connectivity and link failures for consensus algorithms to converge. We started by showing that ordinary consensus and push-sum still work if intercommunication intervals do not grow too fast. Then we moved on to our main result, which is a fully asynchronous push-sum algorithm robust to link failures. We proved its convergence while allowing consecutive link failures to grow to infinity, as long as they remain smaller than a logarithmically growing upper bound.

This work can be extended by improving the upper bounds using ergodicity theory.
It is also possible to use our results to develop asynchronous distributed optimization algorithms robust to packet losses.




\begin{thebibliography}{10}
\small
\providecommand{\url}[1]{#1}
\csname url@samestyle\endcsname
\providecommand{\newblock}{\relax}
\providecommand{\bibinfo}[2]{#2}
\providecommand{\BIBentrySTDinterwordspacing}{\spaceskip=0pt\relax}
\providecommand{\BIBentryALTinterwordstretchfactor}{4}
\providecommand{\BIBentryALTinterwordspacing}{\spaceskip=\fontdimen2\font plus
\BIBentryALTinterwordstretchfactor\fontdimen3\font minus
  \fontdimen4\font\relax}
\providecommand{\BIBforeignlanguage}[2]{{%
\expandafter\ifx\csname l@#1\endcsname\relax
\typeout{** WARNING: IEEEtran.bst: No hyphenation pattern has been}%
\typeout{** loaded for the language `#1'. Using the pattern for}%
\typeout{** the default language instead.}%
\else
\language=\csname l@#1\endcsname
\fi
#2}}
\providecommand{\BIBdecl}{\relax}
\BIBdecl

\bibitem{lorenz2011convergence}
J.~Lorenz, ``Convergence to consensus in multiagent systems and the lengths of
  inter-communication intervals,'' \emph{arXiv preprint arXiv:1101.2926}, 2011.

\bibitem{Blondel2005}
V.~D. Blondel, J.~M. Hendrickx, A.~Olshevsky, and J.~N. Tsitsiklis,
  ``Convergence in multiagent coordination, consensus, and flocking,'' in
  \emph{Decision and Control, 2005 and 2005 European Control Conference.
  CDC-ECC'05. 44th IEEE Conference on}.\hskip 1em plus 0.5em minus 0.4em\relax
  IEEE, 2005, pp. 2996--3000.

\bibitem{charalambous2015distributed}
T.~Charalambous, Y.~Yuan, T.~Yang, W.~Pan, C.~N. Hadjicostis, and M.~Johansson,
  ``Distributed finite-time average consensus in digraphs in the presence of
  time delays,'' \emph{IEEE Transactions on Control of Network Systems},
  vol.~2, no.~4, pp. 370--381, 2015.

\bibitem{kempe2003gossip}
D.~Kempe, A.~Dobra, and J.~Gehrke, ``Gossip-based computation of aggregate
  information,'' in \emph{Foundations of Computer Science, 2003. Proceedings.
  44th Annual IEEE Symposium on}.\hskip 1em plus 0.5em minus 0.4em\relax IEEE,
  2003, pp. 482--491.

\bibitem{hadjicostis2014average}
C.~N. Hadjicostis and T.~Charalambous, ``Average consensus in the presence of
  delays in directed graph topologies,'' \emph{IEEE Transactions on Automatic
  Control}, vol.~59, no.~3, pp. 763--768, 2014.

\bibitem{hadjicostis2012average}
------, ``Average consensus in the presence of delays and dynamically changing
  directed graph topologies,'' \emph{arXiv preprint arXiv:1210.4778}, 2012.

\bibitem{Rezaeinia2017}
P.~Rezaeinia, B.~Gharesifard, T.~Linder, and B.~Touri, ``Push-sum on random
  graphs,'' \emph{arXiv preprint arXiv:1708.00915}, 2017.

\bibitem{benezit2010weighted}
F.~B{\'e}n{\'e}zit, V.~Blondel, P.~Thiran, J.~Tsitsiklis, and M.~Vetterli,
  ``Weighted gossip: Distributed averaging using non-doubly stochastic
  matrices,'' in \emph{Information theory proceedings (isit), 2010 ieee
  international symposium on}.\hskip 1em plus 0.5em minus 0.4em\relax IEEE,
  2010, pp. 1753--1757.

\bibitem{DBLP:journals/corr/GerencserH15}
\BIBentryALTinterwordspacing
B.~Gerencs{\'{e}}r and J.~M. Hendrickx, ``Push sum with transmission
  failures,'' \emph{CoRR}, vol. abs/1504.08193, 2015. [Online]. Available:
  \url{http://arxiv.org/abs/1504.08193}
\BIBentrySTDinterwordspacing

\bibitem{Vaidya}
C.~N. Hadjicostis, N.~H. Vaidya, and A.~D. Dominguez-Garcia, ``Robust
  distributed average consensus via exchange of running sums,'' \emph{IEEE
  Transactions on Automatic Control}, pp. 1492--1507, Jun. 2016.

\bibitem{Schenato}
\BIBentryALTinterwordspacing
N.~Bof, R.~Carli, and L.~Schenato, ``Average consensus with asynchronous
  updates and unreliable communication,'' \emph{IFAC-PapersOnLine}, vol.~50,
  no.~1, pp. 601 -- 606, 2017, 20th IFAC World Congress. [Online]. Available:
  \url{http://www.sciencedirect.com/science/article/pii/S240589631730126X}
\BIBentrySTDinterwordspacing

\bibitem{tsianos2012push}
K.~I. Tsianos, S.~Lawlor, and M.~G. Rabbat, ``Push-sum distributed dual
  averaging for convex optimization,'' in \emph{Decision and Control (CDC),
  2012 IEEE 51st Annual Conference on}.\hskip 1em plus 0.5em minus 0.4em\relax
  IEEE, 2012, pp. 5453--5458.

\bibitem{varagnolo2016newton}
D.~Varagnolo, F.~Zanella, A.~Cenedese, G.~Pillonetto, and L.~Schenato,
  ``Newton-raphson consensus for distributed convex optimization,'' \emph{IEEE
  Transactions on Automatic Control}, vol.~61, no.~4, pp. 994--1009, 2016.

\bibitem{bof2017newton}
N.~Bof, R.~Carli, G.~Notarstefano, L.~Schenato, and D.~Varagnolo,
  ``Newton-raphson consensus under asynchronous and lossy communications for
  peer-to-peer networks,'' \emph{arXiv preprint arXiv:1707.09178}, 2017.

\bibitem{seneta2006non}
E.~Seneta, \emph{Non-negative matrices and Markov chains}.\hskip 1em plus 0.5em
  minus 0.4em\relax Springer Science \& Business Media, 2006.

\bibitem{bremaud2013markov}
P.~Br{\'e}maud, \emph{Markov chains: Gibbs fields, Monte Carlo simulation, and
  queues}.\hskip 1em plus 0.5em minus 0.4em\relax Springer Science \& Business
  Media, 2013, vol.~31.

\bibitem{bertsekas1989parallel}
D.~P. Bertsekas and J.~N. Tsitsiklis, \emph{Parallel and distributed
  computation: numerical methods}.\hskip 1em plus 0.5em minus 0.4em\relax
  Prentice hall Englewood Cliffs, NJ, 1989, vol.~23.

\end{thebibliography}
\end{document}